\documentclass{amsart}
\usepackage{amsfonts}
\usepackage{amsthm}
\usepackage{amsmath}
\usepackage{xcolor} 
\usepackage{amssymb}
\usepackage{mathtools} \usepackage{amscd}
\usepackage{dsfont} \usepackage[latin2]{inputenc}
\usepackage{t1enc}
\usepackage[mathscr]{eucal}
\usepackage{indentfirst}
\usepackage{graphicx}
\usepackage{graphics}
\PassOptionsToPackage{hyphens}{url}
\usepackage{hyperref}
\usepackage{epic}
\numberwithin{equation}{section}
\usepackage{epstopdf} 
\usepackage{enumitem}
\usepackage[sort&compress,numbers]{natbib}

\usepackage{comment}
\numberwithin{equation}{section}
\newtheorem{assumption}{Assumption}[section]
\newtheorem{theorem}{Theorem}[section]
\newtheorem{lemma}[theorem]{Lemma}
\newtheorem{corollary}[theorem]{Corollary}
\newtheorem{definition}[theorem]{Definition}
\newtheorem{proposition}[theorem]{Proposition}

\newtheorem{example}{Example}[section]

\DeclareMathOperator*{\argmin}{arg\,min}

\def\mF{{\mathcal F}}
\def\mP{{\mathcal P}}

\def\mB{{\mathcal B}}
\def\mE{{\mathcal E}}

\def\mM{{\mathcal M}}
\def\mI{{\mathcal I}}
\def\mH{{\mathcal H}}

\def\mS{{\mathcal S}}
\def\mG{{\mathcal G}}
\def\mN{{\mathcal N}}

\def\relu{\mathrm{ReLU}}
\def\sp{\mathrm{SP}}

\def\R{{\mathbb R}}

\def\mZ{{\mathcal Z}}

\def\N{{\mathbb N}}
\def\bP{{\mathbf P}}

\def\bE{{\mathbf E}}
\def\gt{{\rightarrow}}

\def\eps{{\varepsilon}}

 \def\1{{\mathbf 1}}

\newcommand{\abs}[1]{\lvert#1\rvert}

\newcommand{\norm}[1]{\lVert#1\rVert}

\newcommand{\average}[1]{\langle#1\rangle}

\usepackage{todonotes}

\allowdisplaybreaks

\begin{document}

\title[Neural Networks for Schr\"odinger Eigenvalue Problem]{A Priori Generalization Error Analysis of Two-Layer Neural Networks for Solving High Dimensional Schr\"odinger 
Eigenvalue Problems}

\author{Jianfeng Lu}
\address{(JL) Departments of Mathematics, Physics, and Chemistry, Duke University, Box 90320, Durham, NC 27708.}
\email{jianfeng@math.duke.edu}
\author{Yulong Lu}
\address{(YL) Department of Mathematics and Statistics, Lederle Graduate Research Tower, University of Massachusetts, 710 N. Pleasant Street, Amherst, MA 01003.} 
\email{lu@math.umass.edu} 

\thanks{J.L.~is supported in part by National Science Foundation via grants DMS-2012286 and CCF-1934964. Y.L.~is supported by the start-up fund of the Department of Mathematics and Statistics at UMass Amherst.}

%\date{\today}

\begin{abstract}
This paper analyzes the  generalization error of two-layer neural networks for computing the ground state of the Schr\"odinger operator on a $d$-dimensional hypercube. We prove that the convergence rate of the generalization error is independent of the dimension $d$, under the a priori assumption that the ground state lies in a spectral Barron space. We verify such assumption by proving a new regularity estimate for the ground state in the spectral Barron space. The later is achieved by a fixed point argument based on the Krein-Rutman theorem. 
%This paper concerns the a priori generalization analysis of the Deep Ritz Method (DRM) [W. E and B. Yu, 2017], a popular neural-network-based method for solving high dimensional partial differential equations.  We derive the generalization error bounds  of two-layer neural networks in the framework of the DRM for solving two prototype elliptic PDEs: Poisson equation and static Schr\"odinger equation on the $d$-dimensional unit hypercube. Specifically,  we prove that the convergence rates of generalization errors are independent of the dimension $d$, under the a priori assumption that the exact solutions of the PDEs lie in a suitable low-complexity space called spectral Barron space. Moreover, we give sufficient conditions on the forcing term and the potential function which guarantee that the solutions are spectral Barron functions. We achieve this by developing a new solution theory for the PDEs on the spectral Barron space, which can be viewed as an analog of the classical Sobolev regularity theory for PDEs. 
\end{abstract}

\maketitle

\section{Introduction}

High dimensional partial differential equations (PDEs) arise ubiquitously from scientific and engineering problems which involve many degrees of freedom, examples include many-body quantum mechanics, phase space description of chemical dynamics, learning and control of complex systems, spectral methods for high dimensional data, just to name a few. While numerical methods for partial differential equations in low-dimension are quite standard, the numerical solution to high dimensional PDEs has remained an outstanding challenge due to the well-known curse of dimensionality. Namely, the computational cost can grow exponentially as the dimension increases.  Perhaps the most celebrated and important example of such challenge is to determine the ground state of many-body quantum systems, which amounts to solving  eigenvalue problems for high dimensional PDEs. 

In recent years, the artificial neural-networks have shown great success in representing high-dimensional classifiers or probability distributions in a variety of machine learning tasks and have led to the tremendous success and popularity of deep learning \cite{lecun2015deep, schmidhuber2015deep}. Motivated by those recent success, researchers have been actively exploring using deep learning techniques to solve  high dimensional PDEs \cite{weinan2018deep,raissi2019physics,han2018solving,sirignano2018dgm,zang2020weak,gu2020selectnet,chen2020friedrichs} by parametrizing the solution as a neural network, including eigenvalue problems for many-body quantum systems (see e.g., \cite{carleo2017solving, gao2017efficient, cai2018approximating, han2019solving, hermann2020deep, pfau2020ab, choo2020fermionic, han2020solving}).
Despite wide popularity and many successful examples of employing neural network ansatz for solving PDEs, their theoretical analyses  are still sparse. For the recent theoretical development, we refer the interested readers to  \cite{shin2020convergence,shin2020error, mishra2020estimates,hong2021rademacher,xu2020finite,luo2020two}. Nonetheless, to the best of our knowledge, the numerical analysis of neural network methods for high dimensional eigenvalue problems is not yet established. The goal of this paper is to provide an \emph{a priori} generalization analysis of variational methods for  computing the ground state of the Schr\"odinger operator in high dimension based on the two-layer neural network ansatz. 

Our generalization error analysis  follows largely the framework established in our previous work \cite{lu2021priori}, where the a priori generalization error is analyzed for deep Ritz method for solving elliptic equations. In particular, to establish approximation results that do not deteriorate as dimension increases, we will work in the spectral Barron space \cite{barron1993universal,siegel2020high,ma2020towards,lu2021priori}. In fact, it has been shown in those works that Barron functions has  ``lower complexity" than more familiar regularity-based functions such as Sobolev or H\"older functions in the sense that the former  can be efficiently approximated by two-layer neural networks without curse of dimensionality. On the other hand, as the Barron space is rather different from  Sobolev or H\"older spaces, the main challenge one faces is to establish regularity theory for high dimensional PDEs in such space. Our previous work \cite{lu2021priori} established the appropriate solution theory for elliptic equations. The key contribution of the present work is to extend such novel solution theory to Schr\"odinger eigenvalue problems in high dimension. Since we are working in spectral Barron space, which is a general Banach space without inner product structure, Lax-Milgram or Courant-Fisher theorems are not applicable, and thus we have to rely on fixed point theorem to establish the existence of solutions. In our previous work for elliptic PDEs \cite{lu2021priori}, Fredholm alternative principle was used; while in this work, to establish the existence of nontrivial eigenfunctions, we rely on the Krein-Rutman theorem \cite{krein1962linear}. 

%\subsection{Plan of the Paper} 
The remainder of this paper is organized as follows. In Section \ref{sec:main} we first set up the ground state problem of the Schr\"odinger operator and present the main generalization results (see Theorems \ref{thm:gen1}-\ref{thm:gen2}) and the new regularity estimate on the ground state in the spectral Barron space (see Theorem \ref{thm:reg}). In Section  \ref{sec:stab} we prove a key stability estimate on the ground state, which allows us to bound the $H^1$-error between the ground state and its approximation in terms of the energy excess. We present the proof of the main generalization result in Section \ref{sec:proofmain} and the proof of the new  regularity estimate on the ground state in Section \ref{sec:reg}.

\section{Set-Up and Main Results}\label{sec:main}
\subsection{Set-Up of Problem}
Let $\Omega  = [0,1]^d$ be the unit hypercube on $\R^d$ with the boundary $\partial \Omega$.  Consider the Neumann eigenvalue problem for the Schr\"odinger operator 
$$
\begin{aligned}
   & \mH u = -\Delta u + V u=  \lambda u & \text{ in } \Omega, \\
 &  \frac{\partial u}{\partial \nu} = 0 & \text{ on } \partial\Omega.
\end{aligned}
$$
where  $\mH:= -\Delta + V $ is the Schr\"odinger operator with the potential function $V$ equipped with the Neumann boundary condition.  We are particularly interested in computing the ground state of $\mH$, that is the eigenfunction associated to the smallest eigenvalue of $\mH$. 

Throughout the paper we make the following minimum assumption on the potential function. 
\begin{assumption}\label{ass:v}
There exist finite positive constants $V_{\min} $ and $ V_{\max}$ such that  $ V_{\min}\leq V(x) \leq V_{\max}$ for every $x\in \Omega$. 
\end{assumption}
Note that we do not lose generality by assumption that $V_{\min}$ is positive, since one can always add a constant to $V$ without changing the eigenfunctions. 

 It is well-known that the minimum eigenvalue $\lambda_0$ can be characterized as the minimum of Rayleigh quotient, i.e.
\begin{equation}\label{eq:lambda0}
    \lambda_0 = \min_{u\in H^1(\Omega)} \mE(u) :=  \min_{u\in H^1(\Omega)}   \frac{\langle u, \mH u\rangle}{ \langle u,u\rangle}, 
\end{equation}
where $\langle\cdot, \cdot\rangle$ denotes the inner product on $L^2(\Omega)$. Under Assumption \ref{ass:v},  the minimizer of the variational problem \eqref{eq:lambda0} is achieved at the  ground state $u_0\in H^1(\Omega)$. Moreover, the ground state $u_0$ is unique (up to a multiplicative constant) and is strictly positive (up to a global sign) on $\Omega$; see e.g., \cite[Theorem 3.3.2]{glimm1987functional}. Without loss of generality we assume further that the ground state $u_0$ is normalized, i.e., $\|u_0\|_{L^2(\Omega)} =1$.   
%Notice that $\mE(u)$ is scaling-invariant in the sense that $\mE(u) = \mE(\alpha u)$ for any $\alpha \in \R$. 

For certain results to hold, we may also need to make the following additional spectrum assumption on the Schr\"odinger operator $\mH$.

\begin{assumption}\label{ass:spect}
The operator  $\mH$ has discrete spectrum $\{\lambda_j\}_{j=0}^\infty$ with a positive spectral gap, i.e.  $ \lambda_0 < \lambda_1 \leq \lambda_2\leq \cdots \leq \lambda_k \uparrow \infty $.
\end{assumption}
To  avoid any confusion with the subscript in the notation, let us denote the ground state eigenpair by $(\lambda^\ast, u^\ast) := (\lambda_0, u_0)$, which our study focuses on. 

The natural idea is to seek an approximate solution to Problem \eqref{eq:lambda0} within some  hypothesis class $\mF \subset H^1(\Omega)$ that are  parameterized by neural networks. In practice, the Monte-Carlo method is employed to compute the high dimensional integrals defined by the inner products in \eqref{eq:lambda0}, leading to the definition of empirical loss (or risk) minimization. More concretely, let  us denote by $\mP_\Omega$  the uniform probability distribution on the domain $\Omega$. Then the population loss  $\mE$ can be written as 
\begin{equation}\label{eq:e}
\mE(u)  = \frac{\mE_V(u)}{\mE_2(u)} := \frac{\bE_{X\sim \mP_\Omega}\Big[|\nabla u(X)|^2 + V(X)|u(X)|^2\Big]}{\bE_{X\sim \mP_\Omega}\Big[ |u(X)|^2\Big]}. 
\end{equation}
Let $\{X_j\}_{j=1}^n$ be a sequence of random variables that are  independent and identically distributed (i.i.d.)~  according to $\mP_\Omega$. The population loss is approximated by the following  empirical loss
\begin{equation}\label{eq:en}
\mE_n(u)  = \frac{\mE_{n,V}(u)}{\mE_{n,2}(u)} ,
\end{equation}
where $\mE_{n,V}$ and $\mE_{n,2}$ are defined by 
$$\begin{aligned}
   \mE_{n,V} & := \frac{1}{n}\sum_{j=1}^n \big(|\nabla u(X_j)|^2 + V(X_j)|u(X_j)|^2\big)\\
   \mE_{n,2} & := \frac{1}{n}\sum_{j=1}^n |u(X_j)|^2.
\end{aligned}
$$
Note that we have used the fact that $|\Omega|=1$ in deriving the Monte-Carlo approximation above. 
Let $u_n$ be a minimizer of $\mE_n$ within $\mF$, i.e., $u_n = \argmin_{u\in \mF} \mE_n(u)$. Again since  $\mE_n(u)$ is scaling-invariant, we may assume that $\|u_n\|_{L^2} = 1$.  Our goal is to obtain quantitative estimates for the  error between $u_n$ and $u^\ast$, following the statistical learning literature, we will call such error the generalization error. 

We are interested in quantifying the error between $u_n$ and $u^\ast$ in terms of two quantities. The first one is given by the energy excess $\mE(u_n) - \mE(u^\ast)$ that quantifies the approximation of $\mE(u_n^m)$ to the leading eigenvalue $\lambda^\ast = \mE(u^\ast)$. 

To introduce the second quantity for measuring the error, we define the  projection operator $P$ onto the space of ground state by setting
$$
P u = \langle u, u^\ast \rangle u^\ast.
$$
Let us also define the operator  $P^\perp = (I - P)$, i.e. 
$$
P^\perp u = u - \langle u, u^\ast \rangle u^\ast
$$
Notice that if both $u_0$ and $u$ are normalized, then 
$$
\|P^\perp u \|_{L^2}^2  = 1 - |\langle u, u^\ast \rangle|^2.
$$
Therefore  $\|P^\perp u_n\|_{L^2(\Omega)}$ quantifies the offset of the direction of $u_n$  from that of $u^\ast$. 
The following theorem shows that the $H^1$-norm of $P^\perp u $ can be bounded above by the energy excess. 
\begin{proposition}\label{prop:stab}
Assume that $\mH$ satisfies Assumptions \ref{ass:v} and \ref{ass:spect}. Then for any  $u\in H^1(\Omega)$,
\begin{align*}
& \|P^\perp u\|_{L^2(\Omega)}^2 \leq \frac{\mE(u) - \mE(u^\ast)}{\lambda_1 -\lambda^\ast}  \|u\|_{L^2(\Omega)}^2, \\
& \norm{\nabla P^\perp u}_{L^2(\Omega)}^2 \leq  \bigl(\mE(u) - \mE(u^\ast)\bigr) \Bigl(\frac{V_{\max} - V_{\min}}{\lambda_1 - \lambda^\ast}  + 1\Bigr)  \norm{u}_{L^2(\Omega)}^2. 
\end{align*}
In particular, if $\|u\|_{L^2} =1$, then  
$$
\|P^\perp u\|_{H^1(\Omega)}^2 \leq \Bigl(\frac{V_{\max} - V_{\min} + 1}{\lambda_1 - \lambda^\ast}  + 1\Bigr) \bigl(\mE(u) - \mE(u^\ast)\bigr). 
$$
\end{proposition}

\subsection{Main Results}
We aim to prove quantitative generalization error estimates between the approximation ground state $u_n$ parametrized by neural networks and the exact ground state $u^\ast$. Our particular interest is to show that under certain circumstances the generalization error of the neural network solution does not suffer from the curse of dimensionality. To this end, we will assume (and  prove below) that the exact ground state $u^\ast$ lies in a smaller function space than the usual Sobolev space within which the functions can be approximated by neural networks without curse of dimensionality. Specifically, we assume that $u^\ast$ belongs to  the {\em spectral Barron space} \cite{lu2021priori} defined as follows. 

Let us first define the set of cosine functions
$$
\mathscr{C} = \Bigl\{ \Phi_k \Bigr\}_{k\in \N_0^d}:=  \Bigl\{ \prod_{i=1}^d \cos(\pi k_i x_i) \ |\ k_i \in \N_0 \Bigr\}.
$$
Let $\{\hat{u}(k)\}_{k\in \N_0^d}$ be the expansion coefficients of a function $u\in L^1(\Omega)$ under the basis $\{\Phi_k\}_{k\in \N_0^d}$.
 For $s\geq 0$, the spectral Barron space $\mB^s(\Omega)$ on $\Omega$ is defined  by 
\begin{equation}\label{eq:barrons}
    \mB^s(\Omega) := \Big\{u\in L^1(\Omega): \sum_{k\in \N^d_0} (1 + \pi^{s}|k|_1^{s}) |\hat{u}(k)| < \infty \Big\},
\end{equation}
which is equipped with the spectral Barron norm 
$$
\|u\|_{\mB^s(\Omega)}  = \sum_{k\in \N^d_0} (1 + \pi^{s}|k|_1^{s}) |\hat{u}(k)|.
$$
Note that we use $|k|_1$ to denote the $\ell^1$-norm of a vector $k$. 
It is clear that $B^s(\Omega)$ is a Banach space. Moreover, since functions in $B^s(\Omega)$ have summable cosine coefficients. we have $B^s(\Omega) \hookrightarrow C(\overline{\Omega})$. When $s=2$, we adopt the short notation $\mB(\Omega)$ for $\mB^2(\Omega)$.  Our notion of spectral Barron space is an adaption of the Barron space defined in the seminal work \cite{barron1993universal}; see also the recent works \cite{bach2017breaking,klusowski2018approximation,e2019barron,siegel2020approximation} on other variants of Barron spaces. The original Barron function $f$ in \cite{barron1993universal} is defined on the whole space $\R^d$ whose Fourier transform $\hat{f}(w)$ satisfies that $\int |\hat{f}(\omega)| |\omega| d\omega < \infty$. Our spectral Barron space $\mB^s(\Omega)$ with $s=1$, defined on the bounded domain $\Omega$, can be viewed as a finite domain analog of the original  Barron space from  \cite{barron1993universal}. 

Functions in the spectral Barron space differ substantially from those in Sobolev or H\"older spaces; most importantly, they can be approximated with respect to the $H^1$-norm by two-layer neural networks without curse of dimensionality. To make this more precise, we recall an approximation result from \cite{lu2021priori}. Let us define  for an activation function $\phi$, a constant $B>0$ and the number of hidden neurons $m$ the set of functions
 \begin{equation}\label{eq:fphim}
\mF_{\phi, m}(B) := \Big\{c + \sum_{i=1}^m \gamma_i \phi(w_i \cdot x  - t_i), |c|\leq 2B, |w_i|_1=1, |t_i|\leq 1, \sum_{i=1}^m |\gamma_i|\leq 4B\Big\}.
\end{equation} 
We will focus on the rescaled Softplus activation function $$
\sp_{\tau} (z) =  \frac{1}{\tau} \sp(\tau z) = \frac{1}{\tau}\ln(1 + e^{\tau z}),$$
where $\tau>0$ is a rescaling parameter. Observe that $\sp_{\tau} \gt \relu$ pointwisely as $\tau \gt 0$ (see \cite[Lemma 4.6]{lu2021priori}).

Let $\mF_{\sp_{\tau},m}(B)$ be the set of neural networks defined by setting $\phi = \sp_\tau$ in \eqref{eq:fphim}. The following lemma shows that functions in $\mB(\Omega)$ can be well approximated by functions in $\mF_{\sp_{\tau},m}(B)$  without curse of dimensionality. 
\begin{lemma}\label{lem:appro} \cite[Theorem 2.2]{lu2021priori}
For any $u\in \mB(\Omega)$, there exists a two-layer neural network $u_m \in \mF_{\sp_{\tau},m}(\|u\|_{\mB(\Omega)})$ with $\tau  = \sqrt{m}$,
such that 
$$
\|u - u_m\|_{H^1(\Omega)} \leq  \frac{\|u\|_{\mB(\Omega)}(6\log m+30)}{\sqrt{m}}.
$$
\end{lemma}

With the above approximation result at hand, we are ready to state the main generalization theorem as follows.
\begin{theorem}\label{thm:gen1}
Assume that $\mH$ satisfies Assumptions \ref{ass:v} and \ref{ass:spect}.  Assume also that the ground state $u^\ast \in \mB(\Omega)$. Let $u_{n}^m$ be a minimizer of the empirical loss $\mE_n$ within the set $\mF = \mF_{\sp{\tau},m}(B)$ with $B = \|u^\ast\|_{\mB(\Omega)}$ and with $\tau =\sqrt{m}$. Given $\delta \in (0,\frac{1}{3})$, assume that $n$ and $m$ are large enough so that $\xi_i(n,\delta) \leq 1/2,i=1,2,3$ and $\eta(B,m)\leq 1/2$ where  $\xi_i(n,\delta)$ are defined in \eqref{eq:xi1}, \eqref{eq:xi2} and \eqref{eq:xi3}, and $\eta(B,m)$ is defined in \eqref{eq:eta}. 
Then with probability at least $1-3 \delta$, 
$$
\mE(u_n^m) - \mE(u^\ast) \leq  \frac{C_1\ln(\frac{1}{\delta})\sqrt{m (\ln m + 1)}}{\sqrt{n}} + \frac{C_2(\ln m + 1)}{\sqrt{m}},
$$
where $C_1$ depends on $\|u^\ast\|_{\mB(\Omega)},d,V_{\max}$ polynomially and $C_2$ depends on $\|u^\ast\|_{\mB(\Omega)}$ linearly.  In particular, with the choice $m = \sqrt{n}$, we have that there exists $C_3>0$ such that with probability at least $1-3\delta$,
$$
\mE(u_n^m) - \mE(u^\ast) \leq C_3\ln\Big(\frac{1}{\delta}\Big)\cdot n^{-\frac{1}{4}}.
$$
\end{theorem}

The proof of Theorem \ref{thm:gen1}   relies on decomposing the generalization error into the sum of the approximation error (see Section \ref{sec:appro}) and statistical error arising from the Monte-Carlo approximation. The statistical error is further bounded by controlling the Rademacher complexity of certain neural network classes associated to the loss formulation (see Section \ref{sec:stats}). 
Thanks to Proposition \ref{prop:stab}, the generalization error in terms of the energy excess translates directly to that in terms of the $H^1$-norm of $P^\perp u_n^m$. 
\begin{theorem}\label{thm:gen2}
Suppose that the assumption of Theorem \ref{thm:gen1} holds and suppose further that $\mH$ has a spectral gap. Then there exist positive constants $C_4$ and $ C_5$ depending polynomially on $\|u^\ast\|_{\mB(\Omega)},d, V_{\min}, V_{\max}$ and $\lambda_1-\lambda_0$ such that with probability at least $1-3\delta$, 
 $$ 
\|P^\perp u_n^m\|_{H^1(\Omega)}^2 \leq \frac{C_4\ln(\frac{1}{\delta})\sqrt{m (\ln m + 1)}}{\sqrt{n}} + \frac{C_5 \ln(\frac{1}{\delta})(\ln m + 1)}{\sqrt{m}}.
 $$
Setting $m = \sqrt{n}$ in the above leads to that the follow holds for some $C_6>0$ with probability at least $1-3\delta$:
$$
\|P^\perp u_n^m\|_{H^1(\Omega)}^2 \leq C_6 \ln\Big(\frac{1}{\delta}\Big)\cdot n^{-\frac{1}{4}}. 
$$
\end{theorem}

Theorem \ref{thm:gen1} and Theorem \ref{thm:gen2} show that with high probability the convergence rate of the generalization error of two-layer network for approximating the ground state $u^\ast$ and the corresponding leading eigenvalue $\lambda^\ast = \mE(u^\ast)$ does not suffer from the curse of dimensionality provided that the ground state $u^\ast \in \mB(\Omega)$. 

Finally we justify the regularity assumption on the ground state in the following theorem. This gives a novel solution theory of high dimensional eigenvalue problems in Barron type spaces.

 \begin{theorem}\label{thm:reg}
 Assume that $V \in \mB^s(\Omega)$ with $s\geq 0$ and $V$ satisfies Assumption \ref{ass:v}. Then the ground state  $u^\ast \in \mB^{s+2}(\Omega)$. 
 \end{theorem}
Our idea of proving  Theorem \ref{thm:reg} differs from the standard proof of Sobolev regularity of eigenfunctions, which usually relies on bootstrapping  estimates on the  weak derivative  of the eigenfunctions. Instead, we prove Theorem  \ref{thm:reg} by reformulating the ground state problem  as a fixed point problem on the spectral Barron space $\mB^s(\Omega)$; the existence of a nontrivial fixed point is proved by employing the celebrated Krein-Rutman theorem \cite{krein1962linear}.  See Section \ref{sec:reg} for a complete proof.

\section{Stability estimate of the ground state (Proof of Proposition \ref{prop:stab})} \label{sec:stab}
In this section, we show that the offset $\|P^\perp u\|_{L^2(\Omega)}$ of any $u\in H^1(\Omega)$ can be bounded by the energy excess $ \mE(u) - \mE(u^\ast)$.

\begin{proof}
Note that by the spectral gap assumption 
\begin{equation*}
    \mH P^{\perp} - \lambda_0 P^\perp \geq \lambda_1 - \lambda_0 > 0.
\end{equation*}
Let us decompose
\begin{equation*}
  u = P u + P^{\perp} u =: \alpha u_0 + u_{\perp},
\end{equation*}
where $\alpha = \average{u, u_0}$ and $\average{ u_{\perp}, u_0 } = 0$. Substituting above into the Rayleigh quotient, we have
\begin{equation*}
  \begin{aligned}
    \mE(u) & = \frac{\abs{\alpha}^2 \lambda_0 + \average{u_{\perp}, \mH u_{\perp}} }{\norm{u}^2}  \\
    & \geq \frac{\lambda_0 \norm{u}^2 + (\lambda_1 - \lambda_0) \norm{u_{\perp}}^2 }{\norm{u}^2} \\
    & = \lambda_0 + (\lambda_1 - \lambda_0)  \frac{\norm{u_{\perp}}^2}{\norm{u}^2}\\
    & = \mE(u^\ast)+ (\lambda_1 - \lambda_0)  \frac{\norm{u_{\perp}}^2}{\norm{u}^2}.
  \end{aligned}
\end{equation*}
Thus, the $L^2$-norm of $u_{\perp} \equiv P^{\perp} u$ can be bounded as
\begin{equation}\label{eq:l2bound}
  \norm{u_{\perp}}^2 \leq \frac{\mE(u) - \mE(u_0)}{\lambda_1 - \lambda_0} \norm{u}^2.
\end{equation}
Note that this bound cannot be improved as can be seen by taking a linear combination of the first and second eigenstates of $\mH$. 

To obtain the bound on $\nabla u^\perp$, we notice that 
\begin{equation*}
  \begin{aligned}
    \mE(u) = \lambda_0 + \frac{ \average{u_{\perp}, (\mH -\lambda_0) u_{\perp}} }{\norm{u}^2}
  \end{aligned}
\end{equation*}
Thus
\begin{equation*}
  \begin{aligned}
  \average{u_{\perp}, (\mH -\lambda_0) u_{\perp}} & = \norm{\nabla u_{\perp}}_{L^2}^2 + \average{u_{\perp}, (V - \lambda_0) u_{\perp}} \\
&   \leq \bigl(\mE(u) - \mE(u_0)\bigr) \norm{u}^2. 
  \end{aligned} 
\end{equation*}
Rearranging the terms, we arrive at
  \begin{align*}
    \norm{\nabla u_{\perp}}^2 & \leq -  \average{u_{\perp}, (V - \lambda_0) u_{\perp}} + \bigl(\mE(u) - \mE(u_0)\bigr) \norm{u}^2 \\
    & \leq \bigl( V_{\max} - V_{\min} \bigr) \norm{u_{\perp}}^2 + \bigl(\mE(u) - \mE(u_0)\bigr) \norm{u}^2 \\
    & \stackrel{\eqref{eq:l2bound}}{\leq} \bigl(\mE(u) - \mE(u_0)\bigr) \Bigl(\frac{V_{\max} - V_{\min}}{\lambda_1 - \lambda_0}  + 1\Bigr)  \norm{u}^2. \qedhere
  \end{align*} 
\end{proof}

\section{Proof of Theorem \ref{thm:gen1}}\label{sec:proofmain}

\subsection{Oracle Inequality  for the Generalization Error}
We first introduce an oracle inequality for the empirical loss that decomposes the  generalization error into the sum of approximation error and statistical error.  Recall the population loss $\mE$ and the empirical loss $\mE_n$ defined in \eqref{eq:e} and \eqref{eq:en} respectively. Consider the minimization of $\mE_n$ in a function class $\mF$ and we denote by $u_n$ a minimizer of $\mE$ within $\mF$, i.e. $u_n = \argmin_{u\in \mF} \mE_n(u)$.

We aim to bound the energy excess $\Delta \mE_n := \mE(u_n) - \mE(u^\ast)$ where $u^\ast$  is the exact ground state.
Let us first decompose $\Delta \mE_n$  as follows:
\begin{equation}\label{eq:dE}
    \Delta \mE_n = 
    \mE (u_n) 
    - \mE_n(u_n) + \mE_n(u_n) - \mE_n(u_\mF) +
    \mE_n(u_\mF) - 
    \mE(u_\mF) +
    \mE(u_\mF)
    - \mE(u^\ast).
\end{equation}
Here $u_{\mF} = \argmin_{u\in \mF} \mE (u)$. 
Note that $\mE_n(u_n) - \mE_n(u_\mF) \leq 0$ since $u_n$ is the minimizer of $\mE_n$. Therefore 
$$
\begin{aligned}
\Delta \mE_n & \leq \bigl( \mE (u_n) 
    - \mE_n(u_n)\bigr) +
    \bigl( \mE_n(u_\mF) - 
    \mE(u_\mF)\bigr)  +
    \bigl( \mE(u_\mF)
    - \mE(u^\ast) \bigr) =: T_1 + T_2 + T_3.
\end{aligned}
$$
Note that  $T_1$ and $T_2$ are statistical error terms that arises from Monte-Carlo approximation for integration. The third term $T_3$ is the approximation error term due to restricting
    the minimization of $\mE$ from over the set $H^1(\Omega)$ to $\mF$; see an upper bound of $T_3$ in Theorem \ref{thm:energydiff} when $\mF$ is chosen as the set of two-layer neural networks.  To control the statistical errors, we employ the well-known tool of Rademacher complexity, for which we recall its definition as follows. 
    
   \begin{definition}
We define for a set of random variables $\{Z_j\}_{j=1}^n$ independently distributed according to $\mP_{\Omega}$ and a function class $\mS$  the random variable
$$
\hat{R}_n (\mS) := \bE_{\sigma} \Bigl[\sup_{g\in \mS} \Big| \frac{1}{n}\sum_{j=1}^n \sigma_j g(Z_j) \Big| \;\Big|\; Z_1, \cdots, Z_n\Bigr],
$$
where the expectation $\bE_\sigma$ is taken with respect to the independent uniform Bernoulli sequence $\{\sigma_j\}_{j=1}^n$ with $\sigma_j \in \{\pm 1\}$. Then the Rademacher complexity of $\mS$ defined by $R_n (\mS) = \bE_{\mP_\Omega} [\hat{R}_n (\mS) ]$. 
\end{definition}

    Now we are ready to bound $T_1$ and $T_2$ in terms of Rademacher complexities of suitable sets. 
    
 \textbf{Bounding $T_1$.}
Thanks to the scaling-invariance of $\mE$ and $\mE_n$, we can assume that   $u_n$ is normalized, i.e. $\|u_n\|_2=1$. Hence  we have 
$$\begin{aligned}
 T_1 & \leq \biggl\lvert \frac{\mE_{n,V}(u_n)}{\mE_{n,2}(u_n)} - \frac{\mE_{V}(u_n)}{\mE_{2}(u_n)} \biggr\rvert \\
& \leq  \frac{\Big|\mE_{n,V}(u_n) - \mE_{V}(u_n)\Big|}{\mE_{n,2}(u_n)}  +  \frac{\mE_V(u_n)}{\mE_2(u_n)\cdot \mE_{n,2}(u_n)} \Big|\mE_2(u_n) -\mE_{n,2}(u_n)\Big|\\
& =:T_{11} + T_{12}. \\
%& \leq  \sup_{u\in \mF, \|u\|_{L^2}=1} \frac{\Big|\mE_{n,V}(u) - \mE_{V}(u)\Big|}{\mE_{n,2}(u)}  +\sup_{u\in \mF, \|u\|_{L^2}=1} \frac{\mE_V(u)}{\mE_2(u)\cdot \mE_{n,2}(u)} \Big|\mE_2(u) -\mE_{n,2}(u)\Big|
\end{aligned}$$
To bound $T_{11}$ and $T_{12}$, let us define two sets of functions
\begin{equation}
    \begin{aligned}\label{eq:G12}
\mG_1 & := \{g:  g = u^2 \text{ where } u\in \mF\},\\
\mG_2 & := \{g : g = |\nabla u|^2 + V |u|^2 \text{ where } u \in \mF\}.
\end{aligned}
\end{equation}
We assume that the set $\mF$ satisfies  $\sup_{u\in \mF} \|u\|_{L^\infty} \leq M_{\mF} < \infty$ so that $\sup_{g\in \mG_1} \|g\|_{L^\infty} \leq M_{\mF}^2$. Assume further that $ \sup_{g\in \mG_2}\|g\|_{L^\infty} \leq M_{\mG_2}< \infty$. Now 
let us first derive a high-probability lower bound for  $\mE_{n,2}(u_n)$. 
For doing so, we  define for $n\in \N$ and $ \delta>0$ the constant 
\begin{equation}\label{eq:xi1}
    \xi_1(n,\delta)  := 2R_n(\mG_1) + 4M_{\mF}^2\cdot \sqrt{\frac{2\ln(4/\delta)}{n}},
\end{equation}
where $R_n(\mG_1)$ denotes the Rademacher complexity of the set $\mG_1$.
We also define the event 
$$
A_1(n,\delta) := \left\{|\mE_{n,2} (u_n)- \mE_{2} (u_n)| \leq \xi_1(n,\delta)\right\}.
$$
Then applying Lemma \ref{lem:pacgen} to 
$
\mG_1$ we have that  
\begin{equation}\label{eq:e2n}
    \bP \big[A_1(n,\delta) \big] \geq 1-\delta.
\end{equation}
Since by assumption $\mE_2(u_n)=1$,  within the event $A_1(n,\delta)$ we have $\mE_{n,2} (u_n) \geq 1 - \xi_1(n,\delta)$ and hence 
\begin{equation}\label{eq:T1den}
\bP\big[\mE_{n,2} (u_n) \geq 1 - \xi_1(n,\delta)\big] \geq 1 -\delta.
\end{equation}
Notice that  $\mE_V(u_n) \leq M_{\mG_2}$ by the assumption on $\mG_2$. Therefore if $\xi_1(n,d) < 1$, then  
\begin{equation}\label{eq:T12}
    \bP\Big[T_{12} \leq \frac{M_{\mG_2}\cdot \xi_1(n,\delta)}{1- \xi_1(n,\delta)}\Big] \geq \bP\big[A_1(n,\delta)\big]  \geq 1-\delta.
\end{equation}

Next to bound $T_{11}$, let us define the constant \begin{equation}\label{eq:xi2}
\xi_2(n,\delta) := 2R_n(\mG_2) + 4M_{\mG_2} \cdot \sqrt{\frac{2\ln(4/\delta)}{n}}
\end{equation}
and the event 
$$
A_2(n,\delta) := \Big\{ \sup_{u\in \mF} \Big|\mE_{n,V}(u) - \mE_{V}(u)\Big| \leq \xi_2(n,\delta)  \Big\}.
$$
Then applying again Lemma \ref{lem:pacgen} to $\mG_2$ leads to
\begin{equation}\label{eq:T1num}
   \bP \big[A_2(n,\delta)\big]
   \geq 1-\delta.
\end{equation}
As a result of \eqref{eq:T1den} and \eqref{eq:T1num}, one has that  if $\xi_1(n,d) < 1$ then 
\begin{equation}\label{eq:T11}
\bP\Big[T_{11} \leq \frac{\xi_2(n,d)}{1 - \xi_1(n,d)}\Big] \geq \bP\Big[A_1(n,\delta)\cap A_2(n,\delta)\Big] \geq 1-2\delta.
\end{equation}
Therefore it follows from \eqref{eq:T12} and \eqref{eq:T11} that 
\begin{equation}\label{eq:T1}
\bP\left[T_{1} \leq\frac{M_{\mG_2} \cdot \xi_1(n,d) + \xi_2(n,d)}{1 - \xi_1(n,d)}\right] \geq \bP\Big[A_1(n,\delta)\cap A_2(n,\delta)\Big]  \geq 1-2\delta.
\end{equation}

\textbf{Bounding $T_2$.} Similar to the process of bounding $T_1$, by assuming $\|u_{\mF}\|_{L^2} = 1$ we first bound $T_2$ as follows 
$$\begin{aligned}
 T_2 & \leq \biggl\lvert \frac{\mE_{n,V}(u_{\mF})}{\mE_{n,2}(u_{\mF})} - \frac{\mE_{V}(u_{\mF})}{\mE_{2}(u_{\mF})} \biggr\rvert \\
& \leq  \frac{\Big|\mE_{n,V}(u_{\mF}) - \mE_{V}(u_{\mF})\Big|}{\mE_{n,2}(u_{\mF})}  +  \frac{\mE_V(u_{\mF})}{\mE_2(u_{\mF})\cdot \mE_{n,2}(u_{\mF})} \Big|\mE_2(u_{\mF}) -\mE_{n,2}(u_{\mF})\Big|\\
& =:T_{21} + T_{22}. 
\end{aligned}$$
Since $u_{\mF}$ does not depend on the sample points $\{X_i\}$, applying Hoeffding's inequality  from Lemma \ref{lem:hoef} yields that
\begin{equation}\label{eq:E2}
   \bP[A_3(n,\delta)] :=  \bP\Big[|\mE_{n,2}(u_{\mF}) - \mE_2(u_{\mF})| \leq \xi_3(n,\delta)\Big] \geq 1-\delta,
\end{equation}
where 
\begin{equation}\label{eq:xi3} \xi_3(n,\delta):= M_{\mF}^2\cdot  \sqrt{\frac{\ln (2/\delta)}{2n}}.
\end{equation}
Since by assumption $\mE_2(u_{\mF}) = 1$, this implies further that 
\begin{equation}\label{eq:En2}
\bP\Big[\mE_{n,2}(u_{\mF}) \geq 1 -\xi_3(n,\delta)\Big] \geq \bP[A_3(n,\delta)] \geq 1-\delta.
\end{equation}
Combining \eqref{eq:E2} and  \eqref{eq:En2} implies that if $\xi_3(n,\delta) < 1$, then 
$$
\bP\left[T_{22} \leq \frac{M_{\mG_2} \cdot \xi_3(n,\delta)}{1 - \xi_3(n,\delta)}\right]  \geq \bP[A_3(n,\delta)] \geq 1-\delta.
$$
In addition, as a consequence of \eqref{eq:En2} and \eqref{eq:T1num}, we have  
$$
\bP\left[T_{21} \leq \frac{\xi_2(n,\delta)}{1 - \xi_3(n,\delta)}\right]  \geq \bP[A_2(n,\delta)\cap A_3(n,\delta)] \geq 1-2\delta.
$$
Therefore it holds that 
\begin{equation}\label{eq:T2}
\bP\left[T_{2} \leq \frac{\xi_2(n,\delta)+M_{\mG_2} \cdot \xi_3(n,\delta) }{1 - \xi_3(n,\delta)}\right] \geq 1-2\delta.
\end{equation}
\begin{lemma}[Hoeffding]\label{lem:hoef}
Let $Z_1, Z_2, \cdots, Z_n$ be i.i.d. random variables with $a_i\leq Z_i\leq b_i$ a.s. Then for any $t>0$,
$$
\bP \Big(\Big|\frac{\sum_{i=1}^n Z_i}{n} - \bE Z\Big| \geq t\Big) \leq 2 \exp\Big(-\frac{2n^2 t^2}{\sum_{i=1}^n (b_i-a_i)^2}\Big).
$$
\end{lemma}
We recall the following useful PAC-type generalization bound via the Rademacher complexity. 
\begin{lemma} \cite[Theorem 26.5]{shalev2014understanding} \label{lem:pacgen}
Let $Z_1, Z_2, \cdots, Z_n$ be i.i.d. random variables. 
Let $\mG$ be a function class such that $\sup_{g\in \mG} \|g\|_{L^\infty(\Omega)} \leq C_\mG$ and that $\mG$ is symmetric, i.e. $\mG = -\mG$. Then  with probability at least $1-\delta$, 
$$
\sup_{g\in \mG} \Big|\frac{1}{n}\sum_{i=1}^n g(Z_i) - \bE g(Z)\Big|  \leq 2R_n(\mG) + 4C_\mG \sqrt{\frac{2\ln(4/\delta)}{n}}.
$$
\end{lemma}

Combining the bounds derived above leads to the following oracle inequality.  Recall the quantities  $\xi_i(n,\delta) ,i=1,2,3$ defined in \eqref{eq:xi1}, \eqref{eq:xi2} and \eqref{eq:xi3}.

\begin{theorem}\label{thm:oracle}
Let $u_n = \argmin_{u\in \mF} \mE_n(u)$. Let $\delta\in(0,\frac{1}{3})$ be fixed.  Assume that $\xi_i(n,\delta) < 1,i=1,2,3$.

Then  with probability at least $1-3\delta$,
\begin{equation}
    \begin{aligned}\label{eq:oracle}
\mE(u_n) - \mE(u^\ast) & \leq \frac{M_{\mG_2} \cdot \xi_1(n,d) + \xi_2(n,d)}{1 - \xi_1(n,d)} + \frac{M_{\mG_2} \cdot \xi_3(n,\delta) +\xi_2(n,\delta) }{1 - \xi_3(n,\delta)}\\
&  + \bigl( \inf_{u\in \mF} \mE(u) - \mE(u^\ast)\bigr) .
\end{aligned}
\end{equation}
\end{theorem}

\subsection{Bounding the Approximation Error}\label{sec:appro}
  Recall the spectral Barron space $\mB^s(\Omega)$ defined in \eqref{eq:barrons} and the set of two-layer networks $\mF_{\sp_\tau,m}(B)$ defined by setting $\phi = \sp_\tau$ in \eqref{eq:fphim}.  
The following theorem bounds  the approximation error (the third term) in \eqref{eq:oracle} when $u^\ast \in \mB(\Omega)$ and $\mF = \mF_{\sp_\tau,m}$.
 
 \begin{theorem}\label{thm:energydiff}
Let the ground state $u^\ast \in \mB(\Omega)$ with $\|u^\ast\|_{L^2(\Omega)}=1$. Let  $u_m \in \mF_{\sp_\tau,m}(\|u\|_{\mB(\Omega)})$ be defined in  Lemma \ref{lem:appro}. Assume that $V$ satisfies Assumption \ref{ass:v}. Assume in addition that 
\begin{equation}\label{eq:eta}
    \eta(\norm{u^{\ast}}_{\mB(\Omega)}, m) : = \frac{\|u^\ast\|_{\mB(\Omega)}\cdot(6\ln m+30)}{\sqrt{m}} \leq \frac{1}{2}.
\end{equation}
Then 
\begin{equation}
    \mE(u_m) - \mE(u^\ast) \leq \left(2 (1 + V_{\max}) \Big(\sqrt{\frac{\lambda^\ast}{\min(1, V_{\min})}} + 1\Big) + 3\lambda^\ast\right)  \eta(\norm{u^\ast}_{\mB(\Omega)},m)
\end{equation}
 \end{theorem}
 
 \begin{proof}
 By assumption $\mE_2(u^\ast) = \|u^\ast\|^2_{L^2(\Omega)}=1$. Then $\mE(u^\ast) = \mE_V(u^\ast) = \lambda^\ast$. Since $V\geq V_{\min}>0$, this implies that 
 \begin{equation}\label{eq:ustarh1}
 \|u^\ast\|_{H^1(\Omega)}^2 \leq \frac{\mE_V(u^\ast)}{\min(1, V_{\min})}  =  \frac{\lambda^\ast}{\min(1, V_{\min})} .
 \end{equation}
 Now observe that
 \begin{equation}\label{eq:diffe}
 \begin{aligned}
 \mE(u_m) - \mE(u^\ast) & = \frac{\mE_V(u_m)}{\mE_2(u_m)} - 
 \frac{\mE_V(u^\ast)}{\mE_2(u^\ast)}\\
 & = \frac{\mE_V(u_m ) - \mE_V(u^\ast)}{\mE_2(u_m)} + \frac{\mE_2(u^\ast)-
 \mE_2(u_m)}{\mE_2(u_m)}\cdot  \mE(u^\ast).
 \end{aligned}
 \end{equation}
Thanks to Lemma \ref{lem:appro}, $\|u_m - u^\ast\|_{H^1(\Omega)} \leq \eta(\norm{u^\ast}_{\mB(\Omega)},m) < 1$. This implies  that 
$$1-\eta(\norm{u^\ast}_{\mB(\Omega)},m)\leq \|u_m\|_{L^2(\Omega)}  \leq 1+\eta(\norm{u^\ast}_{\mB(\Omega)},m)$$ and that 
$$\begin{aligned}
|\mE_2(u_m) - \mE_2(u^\ast) |& = (\|u^\ast\|_{L^2(\Omega)} + \|u_m\|_{L^2(\Omega)}) | \|u^\ast\|_{L^2(\Omega)} - \|u_m\|_{L^2(\Omega)}|\\
& \leq (2 + \eta(\norm{u^\ast}_{\mB(\Omega)},m)) \eta(\norm{u^\ast}_{\mB(\Omega)},m)\\
& \leq 3 \eta(\norm{u^\ast}_{\mB(\Omega)},m).
\end{aligned}$$
  In addition, it follows from the boundedness of $V$ that
 $$\begin{aligned}
| \mE_V(u_m ) - \mE_V(u^\ast)| & \leq  (1 + V_{\max}) (\|u^\ast\|_{H^1(\Omega)} + \|u_m\|_{H^1(\Omega)}) \|u^\ast - u_m\|_{H^1(\Omega)}\\
& \leq  (1 + V_{\max}) (2\|u^\ast\|_{H^1(\Omega)} +\eta(\norm{u^\ast}_{\mB(\Omega)},m) ) \eta(\norm{u^\ast}_{\mB(\Omega)},m)\\
& \leq  (1 + V_{\max}) \Big(\sqrt{\frac{\lambda^\ast}{\min(1, V_{\min})}} + 1\Big) \eta(\norm{u^\ast}_{\mB(\Omega)},m) ,
 \end{aligned}$$
 where we have used \eqref{eq:ustarh1} in the last inequality. Finally, the estimate  follows simply by substituting the last two estimates into \eqref{eq:diffe}. 
 \end{proof}

\subsection{Bounding the Statistical Error}\label{sec:stats}
In this section we proceed to bound the statistical errors (the first two terms on the right side of \eqref{eq:oracle}). This is achieved by controlling the Rademacher complexities of  the function classes $\mG_1$ and $\mG_2$ defined in \eqref{eq:G12}. More specifically, since we set the trial functions $\mF = \mF_{\sp_\tau,m}$, we need to bound the Rademacher complexities of the following 
\begin{equation}
    \begin{aligned}\label{eq:G12}
    \mG_{\sp_\tau,m,1}(B) & := \big\{g = u^2: u\in \mF_{\sp_\tau,m}(B)\big\},\\
    \mG_{\sp_\tau,m,2}(B) & := \big\{g = |\nabla u|^2 + V |u|^2: u\in \mF_{\sp_\tau,m}(B)\big\}.
    \end{aligned}
\end{equation}
\begin{theorem}\label{thm:rnGsp}
Assume that  $\|V\|_{L^\infty(\Omega)}\leq V_{\max}$.   consider the sets  $\mG_{\sp_{\tau},m,1}(B)$ and  $\mG_{\sp_{\tau},m,2}(B)$ with $\tau = \sqrt{m}$ and $B>0$. Then there exist positive constants $C_{1}(B,d,F)$ and $C_{2}(B,d,V_{\max})$ depending polynomially on $B,d,V_{\max}$  such that
\begin{align}
R_n(\mG_{\sp_{\tau},m,1}(B)) & \leq  \frac{C_{1}(B,d) \sqrt{m}(\sqrt{\ln m}+1)}{\sqrt{n}},\label{eq:mgp}\\
R_n(\mG_{\sp_{\tau},m,2}(B)) & \leq  \frac{C_{2}(B,d,V_{\max}) \sqrt{m}(\sqrt{\ln m}+1)}{\sqrt{n}}.\label{eq:mgs}
\end{align}
\end{theorem}
To prove Theorem \ref{thm:rnGsp}, we rely on the celebrated Dudley's theorem \cite{dudley1967sizes}, which bounds the Rademacher complexities in terms of the metric entropy. Below we restate the Dudley's theorem by following  \cite[Theorem 1.19]{wolfnotes}.

\begin{theorem}[{Dudley's theorem {\cite[Theorem 1.19]{wolfnotes}}}] \label{thm:dudley} Let $\mF$ be a function class such that $\sup_{f\in \mF}\|f\|_{\infty}\leq M$. Then the Rademacher complexity $R_n(\mF)$ satisfies that 
\begin{equation}\label{eq:dudley}
    R_n(\mF) \leq \inf_{0\leq \delta\leq M} \Big\{4\delta + \frac{12}{\sqrt{n}}\int_\delta^M \sqrt{\ln \mN(\eps,\mF, \|\cdot\|_\infty)} \,d\eps\Big\},   
\end{equation}
where $\mN(\eps,\mF, \|\cdot\|_\infty)$ denotes the $\eps$-covering number of $\mF$ w.r.t the $L_\infty$-norm.
\end{theorem}

In order to apply Dudley's theorem, we need to
 bound the $\delta$-covering numbers of the sets $\mG_{\sp_\tau,m,i}(B),i=1,2$. To this end, it will be convenient to introduce the following functions 
\begin{align}
    \mM (\delta,\Lambda,m, d) &:= \frac{4B\Lambda}{\delta}\cdot \Big(\frac{12B\Lambda}{\delta}\Big)^m  \cdot\Big(\frac{3\Lambda}{\delta}\Big)^{dm}\cdot  \Big(\frac{3\Lambda}{\delta}\Big)^m, \label{eq:M}\\
    \mZ (M,\Lambda,d) &:=M\big(\sqrt{(\ln(4B\Lambda))_+} + \sqrt{(\ln(12B \Lambda) + d\ln(3\Lambda)+ \ln(3\Lambda))_+}\big)\label{eq:Z} \\
    &  \qquad + \sqrt{d+3} \int_0^{M} \sqrt{(\ln(1/\eps))_+}d\eps  \nonumber. 
\end{align}

\begin{lemma}\label{lem:bdNg}
Consider the two sets $ \mG_{\sp_\tau,m,i}(B),i=1,2$ defined in  \eqref{eq:G12} with $B>0$. Then  for any $\delta >0$, 
$$
\begin{aligned}
\mN(\delta,  \mG_{\sp_\tau,m,i}(B), \|\cdot\|_{\infty}) & \leq \mM (\delta,\Lambda_i,m, d), \quad i=1,2, 
 \end{aligned}
$$
where the constants $\Lambda_i,i=1,2$ satisfy  that 
\begin{equation}\label{eq:Lambda12}
\Lambda_1 \leq 36B (5 + 8B),  \quad 
\Lambda_2 \leq 64B^2\sqrt{m} + 8B +  36V_{\max}B (5 + 8B).
\end{equation}
\end{lemma}

\begin{proof}
The proof follows directly from  \cite[Lemma 5.5]{lu2021priori} and  \cite[Lemma 5.7]{lu2021priori} by slightly adjusting the constants. We thus omit the details. 
\end{proof}

\begin{proof}[Proof of Theorem \ref{thm:rnGsp}]
First from the definition of $\mF_{\sp_\tau,m}(B)$ and the fact  that $
\|\sp_{\tau}\|_{W^{1,\infty} (\Omega)} \leq 3 + \frac{1}{\tau}
$ (see \cite[Lemma 4.6]{lu2021priori}),
one has the following uniform bound 
$$
 \sup_{u\in \mF_{\sp_\tau,m}(B)} \|u\|_{W^{1,\infty}(\Omega)} \leq 2B + 4B \|\sp_{\tau}\|_{W^{1,\infty}   (\Omega)}  \leq 16B.
$$
This implies that 
$$\begin{aligned}
\sup_{f\in \mF_{\sp_\tau,m}(B)} \|f\|_{L^\infty(\Omega)} & \leq 16B = M_{\mF},\\
\sup_{g\in \mG_{\sp_\tau,m,1}(B)} \|g\|_{L^\infty(\Omega)} &\leq (16B)^2 =: M_1, \\ \sup_{g\in \mG_{\sp_\tau,m,2}(B)} \|g\|_{L^\infty(\Omega)} &\leq (1+ V_{\max})(16B)^2 =: M_2.
\end{aligned}$$
Applying Theorem \ref{thm:dudley} with $ \delta = 0$, we obtain from  Lemma \ref{lem:bdNg} and the estimates in the last line that 
$$
R_n(\mG_{\sp_\tau,m,i}(B)) \leq \mZ (M_i,\Lambda_i,d) \sqrt{\frac{m}{n}}.
$$
Moreover, after plugging the bounds on $\Lambda_i$ (see \eqref{eq:Lambda12}) into $\mZ (M_i,\Lambda_i,d)$ , it is easy to see that
$$\begin{aligned}
\mZ (M_1,\Lambda_1,d) & \leq C_1(B,d),\\
 \mZ (M_1,\Lambda_1,d) & \leq C_2(B,d, V_{\max}) \sqrt{\ln m},
\end{aligned}$$
where the positive constants $C_1(B,d)$. and $C_2(B,d, V_{\max})$ depend on $B$ and $d$ polynomially. Combining the estimates finishes the proof of Theorem \ref{thm:rnGsp}.
\end{proof}

%To construct a cover of $\mG^1_{\sp_\tau,m}(B)$, we build covers of the parameter space. Specifically, we consider  the larger set $\Theta = [-2B,2B]\times B_{1}^m(4B)\times \bigl( B_{2}^d(1) \bigr)^m\times [-1,1]^m$ containing the parameter space associated to \eqref{eq:fspm}. The set $\Theta$ is endowed with the  metric $\rho$ defined  for $\theta = (c, \gamma, w, t), \theta^\prime =(c^\prime, \gamma^\prime, w^\prime, t^\prime)\in \Theta$ by

\subsection{Proof of Theorem \ref{thm:gen1}}
The proof follows directly by combining Theorem \ref{thm:oracle}, Theorem \ref{thm:energydiff} and Theorem \ref{thm:rnGsp}.

\section{Regularity of the Ground State of Schr\"odinger Operator in the spectral Barron Space (Proof of Theorem \ref{thm:reg})}\label{sec:reg}
In this section we aim to prove the regularity  of the ground state $u^\ast$ in the spectral Barron space  as shown in Theorem \ref{thm:reg}. Since our proof relies heavily on the spectrum theory of positive linear operators on ordered Banach spaces (especially the  Krein-Rutman Theorem), we first recall some relevant terminologies and useful facts from linear functional analysis.  

\subsection{A simple lemma on spectral radius}
Let $E$ be a Banach space. Given a bounded linear operator $T$ on $E$, we recall the resolvent set of $T$ defined by 
$$
\rho_E(T) :=\{\lambda \in \mathbb{C}\ |\ (\lambda I - T) \text{ is bijective on } E\}.
$$
The spectrum  $\sigma_E(T)$ of $T$ is the set $\mathbb{C} \setminus \rho_E(T)$, where we have used the subscript $E$ to indicate the dependence on the Banach space $E$. We further denote by $\sigma_{p, E}(T)$ the point spectrum of $T$, i.e.
$$
\sigma_{p,E}(T) = \{\lambda \in \mathbb{C}\ |\ \text{ there exists } v\in E \text{ and }v \neq 0 \text{ such that }(\lambda I - T)v = 0\} \subset \sigma_E(T),
$$
where we call $\lambda \in \sigma_{p,E}(T)$ an eigenvalue of $T$ and $v$ an eigenvector of $T$. The spectral radius of $T$ is given by 
$$
r_E(T) := \sup\{|\lambda| \ | \ \lambda \in \sigma_E(T) \}.
$$ 
Recall the Beurling-Gelfand's formula
$r_E(T) = \lim_{n\gt \infty} \|T^{n}\|^{1/n}$. We also note that if $T:E\gt E$ is compact, then $\sigma_E(T)\setminus \{0\} = \sigma_{p,E}(T)\setminus \{0\}$.

\begin{lemma}\label{lem:specrad}
Let $T:E\gt E$ be a linear compact operator on a Hilbert  space $E$ equipped with an inner product $(\cdot,\cdot)_E$ and the associated norm $\|\cdot\|_E$. Let $F\subsetneqq E$ be a dense subspace of $E$. Assume that $F$ is a Banach space equipped with the norm $\|\cdot\|_F$ and  that $T:F\gt F$ is also compact.  Then $r_E(T) = r_F(T)$. 
\end{lemma}

\begin{proof} First thanks to $F\subset E$, we claim that  $r_E(T) \geq r_F(T)$. In fact, if $r_{F}(T) = 0$, then this holds trivially. In addition, if $r_E(T) = 0$, then $r_{F}(T) = 0$. Indeed, if assume otherwise that $r_F(T) >0$, then since $\sigma_E(T)\setminus \{0\} = \sigma_{p,E}(T)\setminus \{0\}$ due to the compactness of $T$, there exists $\lambda \neq 0$ and $v\in F\setminus \{0\}$ such that $(\lambda I - T)v = 0$. Since $F\subset E$, we have $v\in E$ and hence $\lambda \in \sigma_{E}(T)$. This contradicts with the assumption that $r_E(T) = 0$ and thus proves $r_F(T) = 0$. Finally, assume that both $r_E(T)$ and $r_F(T)$ are positive. We claim that $\sigma_{F}(T)\setminus \{0\} \subset \sigma_{E}(T)\setminus \{0\} $ which immediately implies that $r_E(T)\geq r_F(T)$.  In fact, for any $\lambda\in \sigma_{F}(T)\setminus \{0\} =  \sigma_{p,F}(T)\setminus \{0\}$, there exists $v\in F\setminus \{0\}$ such that $(\lambda I - T)v=0$. Since $F\subset E$, we have $\lambda \in \sigma_{E}(T) \setminus \{0\}$. 
 
Next we show that $r_E(T) = r_F(T)$. Assume otherwise that $r_E(T) > r_F(T)$. By the definition of $r_E(T) $ and the assumption that $T: E\gt E$ is a compact operator, there exists an eigenvalue $\lambda \in \rho_{p,E}(T) $ such that $r_E(T)\geq |\lambda| > r_F(T)$. This implies that $\text{Ker}_F (\lambda I - T)  := \{u\in F : (\lambda I -T)u = 0\} = \{0\}$, i.e., $\text{Ker}_E(\lambda I -T) \cap F = \{0\}$. Since $F$ is a dense subset of $E$, this implies that $\text{Ker}_E(\lambda I -T) \subset \bar{F}^{\perp} = E^{\perp} = \{0\}$. This contradicts with the fact that $\lambda$ is an eigenvalue of $T$ on $E$ and completes the proof of the lemma.
\end{proof}

\subsection{Krein-Rutman theorem and the leading eigenvalue}
In this section, we recall the famous Krein-Rutman Theorem \cite{krein1962linear} on the leading eigenvalue and eigenfunction of  positive operators on ordered Banach spaces. To this end, let us first recall some terminologies on ordered Banach spaces.  Given a Banach space $E$, a closed convex subset $K\subset E$ is called a  {\em cone} on $E$ if $\alpha K \subset K$ for every $\alpha>0$ and $K\cap \{-K\} = \{0\}$. A cone $K$ induces a natural partial ordering $\leq$ on the Banach space $E$: $x\leq y $ if and only if $ y - x \in K$. Therefore a Banach space $E$ with a cone $K$ is called an  {\em ordered Banach space}, denoted by $(E,K)$. If   the cone $K$ satisfies that $\overline{K - K} = E$, then the cone $K$ is called a {\em total} cone. We define $\dot{K} = K\setminus \{0\}$ and denote by $K^\circ$ the interior of $K$.  If $K$ has nonempty interior $K^\circ$, then $K$ is called a  {\em solid} cone. It is not hard to see that a solid cone is  total. 

\begin{example}
Consider the Banach space $C(\overline{\Omega})$ of continuous functions on a bounded domain $\Omega \subset \R^d$. The space $C(\overline{\Omega})$ is an ordered Banach space with cone $C_+(\overline{\Omega})$ consisting  of nonnegative functions in $C(\overline{\Omega})$. This cone is solid since any strictly positive function is an interior point. 
%Let operator $\mT = (-\Delta)^{-1}$ denote the inverse of the negative Laplacian equipped with the Dirichlet boundary condition on $\Omega$. Then $\mT$ 
\end{example}

Consider two ordered Banach spaces $E$ and $F$, with cones $P$ and $Q$ respectively. A linear operator $T: E\gt F$ is called {\em positive} if $T(P) \subset Q$, and {\em strictly positive} if $T(\dot{P}) \subset \dot{Q}$. If in addition $K$ is solid, then $T$ is called {\em strongly positive} if $T(\dot{P}) \subset Q^\circ$.

\begin{theorem}[Krein-Rutman] Let $E$ be an ordered Banach space with a total cone $K\subset E$. Let $T:E\gt E$
be a linear compact positive operator with  spectral radius $r(T) >0$. Then $r(T)$ is an eigenvalue of $T$ and of the dual $T^\ast$ with corresponding eigenvectors $u\in K\setminus \{0\}$ and $u^\ast\in K^\ast \setminus \{0\}$.
\end{theorem}

As an important consequence of the Krein-Rutman theorem, the following theorem establishes the simplicity of the leading eigenvalue of a strongly positive compact operator on an ordered Banach space. 

\begin{theorem}[{\cite[Theorem 3.2]{amann1976fixed}}]\label{thm:KR2}
Let $E$ be an ordered Banach space with a solid cone $K$. Let  $T:E\gt E$
be a  strongly positive compact operator. Then 

(i) The spectral radius $r(T) >0$;

(ii) $r(T)$ is  a simple eigenvalue with an eigenvector $u\in K^{\circ}$ and there is no other eigenvalue with a positive eigenvector.

\end{theorem}

\subsection{Regularity of Ground State}
Consider the Schr\"odinger operator $\mH = -\Delta + V$. Since $\inf_{x} V(x) \geq V_{\min} > 0$, the standard Sobolev estimate implies that the inverse of $\mH$ (with respect to the Neumann boundary condition), denoted by $\mS := \mH^{-1}$, is bounded from $L^2(\Omega) $ to $H^2(\Omega)$ and hence compact on $L^2(\Omega) $. Moreover, $\mS$ has countable many  eigenvalues $\{\mu_{j}\}_{j=0}^\infty$ with $\mu_j \downarrow 0$ as $j\gt\infty$ and with $r(\mS) =\mu_0 =  \frac{1}{\lambda_0}$.

Recall the spectral Barron space $    \mB^s(\Omega)$ defined in \eqref{eq:barrons}.
 We also recall from \cite{lu2021priori} the next important lemma which shows that the operator $\mS : \mB^s(\Omega) \gt \mB^{s+2}(\Omega)$ is bounded.

\begin{lemma}[{\cite[Theorem 6]{lu2021priori}}]\label{lem:bdT}
Assume that $V \in \mB^s(\Omega)$ with $s\geq 0$ and $\inf_{x\in \Omega} \geq V_{\min} >0$. Then the operator $\mS :  \mB^s(\Omega) \gt  \mB^{s+2}(\Omega)$ is bounded. 
\end{lemma}

Notice that the inclusion $\mathcal{I}: \mB^{s+2}(\Omega) \hookrightarrow \mB^{s}(\Omega)$ is compact. In fact, by definition the space $\mB^{s}(\Omega)$ can be viewed as a weighted $\ell^1$ space $\ell^1_{W_s}(\N_0^d)$ of the cosine coefficients defined on the lattice $\N_0^d$ with the weight $W_s(k) = (1 + \pi^s|k |_1^s)$. Therefore the inclusion satisfies that 
$$
\|\mI u\|_{\mB^{s}(\Omega)} = \sum_{k\in \N_0^d} W_{s} (k) |\hat{u}(k)| = \sum_{k\in \N_0^d} \frac{W_{s}(k)}{W_{s+2}(k)} W_{s+2}(k) |\hat{u}(k)|. 
$$
Since $\frac{W_{s}(k)}{W_{s+2}(k)}  \gt 0$ as $|k| \gt \infty$, by a similar argument as used in the proof of \cite[Lemma 7.2]{lu2021priori}, one can conclude that $\mI$ is  compact from  $\ell^1_{W_{s+2}}(\N_0^d)$ to $\ell^1_{W_s}(\N_0^d)$ and hence from $\mB^{s+2}(\Omega)$ to $\mB^{s}(\Omega)$. 
The following corollary is then a direct consequence of Lemma \ref{lem:bdT} and compactness of the inclusion $\mI$ from $\mB^{s+2}(\Omega)$ to $\mB^{s}(\Omega)$.
\begin{corollary}\label{cor:compact}
 Under the same assumption of Lemma \ref{lem:bdT}, the operator $\mS: \mB^s(\Omega) \gt \mB^s(\Omega)$ is compact.
\end{corollary}
 
 Note that by definition, we have $\mB^s(\Omega) \hookrightarrow C(\overline{\Omega})$. Let us define the  cone $\mB^s_+(\Omega)$ of $\mB^s(\Omega)$ by setting
 $$
\mB^s_+(\Omega) := \big\{f\in \mB^s(\Omega): f\geq 0\big\}
.$$
 \begin{lemma}\label{lem:positive}
 The operator $\mS:\mB^s(\Omega) \gt \mB^s(\Omega)$ is strongly positive, i.e. for any non-zero $f\in \mB^s_+(\Omega)$, we have that $\mS f(x) >0$ for all $x\in \overline{\Omega}$. 
 \end{lemma}
 \begin{proof}
 Let $f\in \mB^s_+(\Omega) \subset C(\overline{\Omega})$. 
 For any fixed $t>0$ and $x\in \Omega$, it follows by the Lie-Trotter splitting that 
 $$
 e^{-t \mH} f(x) = \lim_{n\gt \infty}  [e^{-\frac{t}{n}V } e^{\frac{t}{n}\Delta } ]^{n}f(x). 
 $$
 Since $0<V_{\min} \leq V\leq V_{\max}<\infty$, we have 
 $
 e^{-s V_{\max}} g \leq e^{-sV} g \leq   e^{-s V_{\min}} g
 $ for any non-negative $g\in L^\infty(\Omega)$. Thanks to the fact that the heat semigroup $e^{t\Delta}$ is  positivity-preserving,  this implies that 
\begin{equation}\label{eq:semibd}
 e^{-t V_{\max}} e^{t\Delta } f(x) \leq  e^{-t \mH} f(x) \leq  e^{-t V_{\min}} e^{t\Delta } f(x).
 \end{equation}
 Moreover, as a result of the semigroup property, the solution operator $\mS$ can be written as   
 $$
 \mS f(x) =\int_0^\infty e^{-t \mH} f(x) dt.
 $$
 Note that owing to the upper bound of \eqref{eq:semibd} the  integral is finite.   
 It follows from the last identity and the lower bound of \eqref{eq:semibd} that
 $$\begin{aligned}
    \mS f(x) & \geq \int_0^\infty e^{-t V_{\max}} e^{t \Delta} f(x) dt \geq \int_1^2 e^{-tV_{\max}} e^{t\Delta} f(x) dt.
 \end{aligned}
 $$
 Now since $f$ is non-negative, continuous and non-zero on $\Omega$, there exists a set $A\subset \Omega$ with $\text{Leb}(A)>0$ and a constant $c>0$ such that $f \geq c>0$. Thanks to the Gaussian lower bound of the Neumann heat kernel estimate (see e.g. \cite[Theorem 3.4]{bass1991some}), there exist positive constants $c_1$ and $c_2$ such that
 $$\begin{aligned}
 e^{t \Delta} f(x) &= \int_{\Omega} p_t(x,y) f(y)dy \\
 & \geq \int_{\Omega}\frac{c_1}{t^{\frac{d}{2}}} e^{-\frac{|x-y|^2}{c_2t}} f(y)dy\\
 & \geq \frac{c\,  c_1}{t^{\frac{d}{2}}} e^{-\frac{\text{diam}(\Omega)}{c_2 t}}\text{Leb}(A) >0. 
 \end{aligned} 
 $$
 Multiplying above with $e^{-t V_{\max}}$ and then integrating on $[1,2]$ with respect to $t$ yields that 
 \begin{equation*}
 \mS f(x) \geq c\, c_1\text{Leb}(A) \int_1^2 \frac{1}{t^{\frac{d}{2}}} e^{-\frac{\text{diam}(\Omega)}{c_2 t}}dt>0. \qedhere
 \end{equation*}
 \end{proof}

 Now with the above preparations we are ready to present the proof of  Theorem \ref{thm:reg}.
 \begin{proof}[Proof of Theorem \ref{thm:reg}]
 It is clear that the ground state $u^\ast$ of $\mH$ is identical to the eigenfunction  that corresponds to  the spectral radius $r(\mS) =1/\lambda^\ast$ of the inverse operator $\mS = \mH^{-1}$. In order to show that $u^\ast\in \mB^{s+2}(\Omega)$, it suffices to show that $u^\ast \in \mB^s(\Omega)$. In fact, since $(u^\ast,\lambda^\ast)$ solves the Neumann eigenvalue problem 
 $$ 
 \mH u^\ast = -\Delta u^\ast + V u^\ast = \lambda^\ast u^\ast,
 $$
 we have $u^\ast = \lambda^\ast \mS u^\ast$. An application of Lemma \ref{lem:bdT} implies that $u^\ast \in \mB^{s+2}(\Omega)$ if and only if $u^\ast \in \mB^s(\Omega)$. To show the latter,
 let us consider the operator $\mS$ defined on the ordered Banach space $\mB^s(\Omega)$ with the solid cone $\mB^s_+(\Omega)$.   Observe that $\mS:L^2(\Omega) \gt L^2(\Omega)$ is compact and that by Corollary \ref{cor:compact} $\mS: \mB^s(\Omega) \gt \mB^s(\Omega)$ is also compact. Therefore by Lemma \ref{lem:specrad}  the spectral radii of $\mS$ are identical when viewed as  operators on $\mB^s(\Omega)$ and  $L^2(\Omega)$ respectively. It follows from Theorem \ref{thm:KR2} and the strongly positivity of $\mS$ established in Lemma \ref{lem:positive} that there exists a unique (up to a multiplicative constant) eigenfunction $u^\ast\in \mB^s(\Omega)$ corresponding to the spectral radius $r(\mS) = 1/\lambda^\ast$. Moreover, $u^\ast$  is strictly positive on $\overline{\Omega}$. This completes the proof. 
 \end{proof}

\bibliographystyle{plain}
\bibliography{references}

\end{document}